\documentclass{amsart}
\author{}
\date{}
\usepackage[]{latexsym}
\usepackage[]{amssymb}
\usepackage[]{amsmath}
\usepackage{amsfonts}
\usepackage{amsthm}
\usepackage[all]{xy}
\usepackage[]{mathrsfs}
\usepackage{enumerate}
\usepackage{color}
\usepackage{txfonts}

\newtheorem{theorem}{Theorem}[section]

\newtheorem{thm}[theorem]{Theorem}
\newtheorem{cor}[theorem]{Corollary}
\newtheorem{prop}[theorem]{Proposition}
\newtheorem{lemma}[theorem]{Lemma}

\theoremstyle{definition}
\newtheorem{ex}[theorem]{Example}
\newtheorem{question}[theorem]{Question}
\newtheorem{remark}[theorem]{Remark}

\DeclareMathAlphabet{\mathpzc}{OT1}{pzc}{m}{it}

\begin{document}

\newcommand{\nat}{{\mathbb{N}}}
\newcommand{\Nrd}{{\mathrm{Nrd}}}
\newcommand{\tr}{{\mathrm{tr}}}
\newcommand{\id}{{\mathrm{id}}}
\newcommand{\Ad}{{\mathrm{Ad}}}
\newcommand{\rad}{{\mathrm{rad}}}
\newcommand{\pr}{{\mathrm{pr}}}
\newcommand{\Ker}{{\mathrm{Ker}}}
\newcommand{\wedges}[1]{\d b_1\wedge\ldots\wedge \d b_{#1}}
\newcommand{\rank}{{\mathrm{rank}}}
\renewcommand{\dim}{{\mathrm{dim}}}
\newcommand{\coker}{{\mathrm{Coker}}}
\newcommand{\End}{{\mathrm{End}}}
\newcommand{\Hom}{{\mathrm{Hom}}}
\newcommand{\ad}{{\mathrm{ad}}}
\newcommand{\rk}{{\mathrm{rk}}}
\newcommand{\Mon}{{\mathrm{Mon}}}
\newcommand{\Alt}{{\mathrm{Alt}}}
\newcommand{\Trd}{{\mathrm{Trd}}}
\newcommand{\Sym}{{\mathrm{Sym}}}
\newcommand{\Skew}{{\mathrm{Skew}}}
\newcommand{\sesq}{{\mathrm{Sesq}}}
\newcommand{\sep}{{\mathrm{Trd_{sep}}}}
\newcommand{\genquad}{{\mathrm{Q}}}
\newcommand{\herm}{{\mathrm{Herm}}}
\renewcommand{\epsilon}{\varepsilon}
\renewcommand{\geq}{\geqslant}
\renewcommand{\leq}{\leqslant}
\newcommand{\an}{{\mathrm{an}}}
\newcommand{\alt}{{\mathrm{alt}}}
\newcommand{\Int}{{\mathrm{Int}}}
\renewcommand{\Im}{{\mathrm{Im}}}
\newcommand{\qf}[1]{\mbox{$\langle #1\rangle $}}
\newcommand{\pff}[1]{\mbox{$\langle\!\langle #1
\rangle\!\rangle_b $}}
\newcommand{\pfr}[1]{\mbox{$\langle\!\langle #1 ]]$}}
\newcommand{\qp}[2]{\mbox{$[ {#1}\,|\!|\,{#2} )$}}
\newcommand{\HH}{{\mathbb H}}
\newcommand{\ZZ}{{\mathbb Z}}
\newcommand{\NN}{{\mathbb N}}
\newcommand{\FF}{{\mathbb F}}
\newcommand{\mg}[1]{#1^{\times}}
\newcommand{\sq}[1]{#1^{\times 2}}
\newcommand{\scg}[1]{\mg{F}/\sq{F}}
\newcommand{\s}{\sigma}
\newcommand{\lra}{\longrightarrow}
\newcommand{\qi}{\mid}
\newcommand{\qil}{\,\cdot\!\!\mid}
\newcommand{\qir}{\mid\!\!\cdot\,}
\newcommand{\qilr}{\,\cdot\!\!\mid\!\!\cdot\,}
\newcommand{\mf}{\mathfrak}
\newcommand{\qeuni}[1]{[#1)}

\newcommand\ideal[1]{{\left<#1\right>}}
\newcommand\sg[1]{{\left<#1\right>}}

%

%

\title{Total linkage of quaternion algebras and  {P}fister forms in characteristic two}

\thanks{
The first author was supported by Wallonie-Bruxelles International.
 The second author is supported by the {Deutsche Forschungsgemeinschaft} project \emph{The Pfister Factor Conjecture in characteristic two} (BE 2614/4).
 The third author acknowledges the support of the French Agence Nationale de la Recherche (ANR) under reference ANR-12-BL01-0005.}
\author{Adam Chapman}
\email{adchapman@math.msu.edu}
\address{Department of Mathematics, Michigan State University, East Lansing, MI 48824}
\author{Andrew Dolphin}
\email{Andrew.Dolphin@uantwerpen.be}
\address{Departement Wiskunde--Informatica, Universiteit Antwerpen, Belgium}
\author{Ahmed Laghribi}
\email{ahmed.laghribi@univ-artois.fr}
\address{Facult\'e des Sciences Jean Perrin, Laboratoire de math\'ematiques de Lens EA 2462, rue Jean Souvraz - SP18, 62307 Lens, France}

\begin{abstract}
We study the subfields of quaternion algebras that are quadratic extensions of their center in characteristic $2$.
We provide examples of the following: two non-isomorphic quaternion algebras that share all their quadratic subfields, two quaternion algebras that share all their inseparable  but not all their separable quadratic  subfields and two algebras that share all their separable but not all their inseparable quadratic  subfields.
We also discuss quaternion algebras over global fields and fields of Laurent series over a perfect field of characteristic $2$ and show that the quaternion algebras over these fields are determined by their separable quadratic  subfields. 
Throughout, these linkage questions are treated in the more general setting by considering the linkage of Pfister forms. 

%
%


\medskip\noindent
\emph{Mathematics Subject Classification (MSC 2010):} 11E81; 11E04, 16K20, .

\end{abstract}

\maketitle

\section{Introduction}

Let $F$ be a field. It is well known that a quadratic field extension of $F$ is a splitting field of some quaternion $F$-algebra $Q$ if and only if it is isomorphic to a subfield of $Q$ (see, for example, \cite[\S 14]{Draxl:1983}). This raised the question of to what extent do the quadratic extensions of $F$ contained in  $Q$ determine the structure of the algebra and motivated the study of `linked' quaternion algebras, that is, quaternion algebras that share a common quadratic extension of $F$ as a subfield. 

In \cite{GaribaldiSaltman:2010}, Garibaldi and Saltman studied the subfields of quaternion algebras over fields of characteristic not $2$, and in particular over number fields. They gave an example of two non-isomorphic quaternion algebras that share all their quadratic subfields and,  conversely, gave a sufficient condition on a field for every quaternion algebra over that field to be determined by its quadratic subfields.

Here we consider this question in characteristic $2$. Over fields of characteristic $2$ this question is complicated by the fact that two types of quadratic extension are possible, separable and inseparable, and every quaternion algebra contains quadratic subfields of each type.
One can show that if two quaternion algebras share an inseparable quadratic  subfield then they share a separable quadratic subfield, but that the converse is not always true (see \cite{Lam:char2quat}. This result  was also generalized to Hurwitz algebras in \cite{ElduqueVilla:2005}).
 This motivated the consideration of two different types of `linkage' in characteristic $2$, depending on whether the shared subfield of two quaternion algebras over $F$ was a quadratic separable or a quadratic inseparable extension of $F$.

Due to the well-known correspondence between quaternion algebras and $2$-fold Pfister forms, via the norm form of the quaternion algebra, linkage questions on a pair of quaternion algebras can be naturally reinterpreted as questions on whether a pair of $2$-fold Pfister forms become hyperbolic over a common quadratic extension. It is therefore natural to ask whether analogous  linkage properties hold for any pair of $n$-fold Pfister forms.

In this article, for any integer $n>1$, we show how to construct $n$-fold Pfister forms over fields of characteristic $2$ that are not isometric but become hyperbolic over all the same quadratic extensions {of the base field}. We also show how to construct $n$-fold Pfister forms that become hyperbolic over the same  inseparable quadratic extensions but not all the same separable quadratic extensions, and $n$-fold Pfister forms that become hyperbolic over all the same separable quadratic extensions but not all the same inseparable quadratic extensions. This  shows  in particular that the  linkage result from \cite{Lam:char2quat} does not have a natural analogue if we consider Pfister forms  becoming hyperbolic over all  quadratic  subfields of a certain type. We shall also discuss the translation of these results in the case of $2$-fold Pfister forms into results on quaternion algebras. 

Finally, we consider Pfister forms over fields with a unique inseparable quadratic extension. 
We show that there exist fields with this property such that there are  $2$-fold Pfister forms defined on  them  that  become hyperbolic over all the same quadratic subfields, yet still fail to be isometric. However, we also show that for  
 global fields and fields of Laurent series over a perfect field in characteristic $2$, $2$-fold Pfister  forms are determined by the separable quadratic extensions that they become hyperbolic over.

The collaboration between the authors started as a result of the workshop on ``Torsors, Motives and Cohomological Invariants" held at Fields Institute, Toronto in May 2013. We wish to thank Pasquale Mammone for bringing up the main research question. Thanks also to Adrian Wadsworth for many useful comments  on an earlier version of this article, and in particular for a simplified argument for Theorem \ref{thm:global}.


\section{Preliminaries on Quadratic forms}
Throughout this article, let $F$ be a field of characteristic $2$. We  denote the multiplicative group of $F$ by  $F^\times$ and the additive group $\{x^2+x\mid x\in F\}$ by $\wp(F)$.
We recall the basic definitions and results we use from the theory of quadratic forms over fields. We refer to \cite{Elman:2008} as a general reference.

A  \emph{bilinear form over $F$} is a pair $(V,b)$ where $V$ is a finite dimensional $F$--vector space and $b$ is a  $F$--bilinear map  $b:V\times V\rightarrow F$.  
 The \emph{radical of $(V,b)$} is the set $$\rad(V,b)=\{x\in V\mid b(x,y)=0 \textrm{ for all } y\in  V\}\,.$$
 We say that $(V,b)$ is
\emph{nondegenerate} if $\rad(V,b)= \{0\}$ and  \emph{degenerate} otherwise. 
We say  that $(V,b)$  is \emph{symmetric} if $b(x,y)=b(y,x)$ {for all $x,y\in V$}.

Let $\delta=(V,b)$ and $\eta=(W,b')$ be two symmetric bilinear forms over $F$.
By an \emph{isometry of bilinear forms $f:\delta\rightarrow\eta$} we mean an isomorphism of $F$--vector spaces $f:V\rightarrow W$ such that $b(x,y)=b'(f(x),f(y))$ for all $x,y\in V$. If such an isometry exists, we say $\delta$ and $\eta$ are \emph{isometric} and we write $\delta\simeq\eta$. 
The \emph{tensor product of $\delta$ and $\eta$} is defined to be the pair $(V\otimes W,b'')$ where the $F$--bilinear map $b'':(V\otimes W)\times (V\otimes W)\rightarrow F$ is given by $b''(v_1\otimes w_1,v_2\otimes w_2)= b(v_1,v_2) \cdot b'(w_1,w_2)$ for all $v_1, v_2\in V$ and $w_1,w_2\in W$, and we write  $\delta\otimes \eta= (V\otimes W, b'')$.

For $a_1,\ldots,a_n\in F^\times$, let $(F^n,b)$ be  the symmetric bilinear form $(F^n,b)$ where  $$b:F^n\times F^n\rightarrow F, \qquad((x_1,\ldots, x_n),(y_1,\ldots, y_n))\mapsto \sum_{i=1}^n a_ix_iy_i .$$
We   denote the symmetric bilinear  form $(F^n,b)$  by $\qf{a_1,\ldots,a_n}_b$.
For $a\in F^\times$ we denote $\qf{1,a}_b$  by $\pff{a}$.
For  $a_1,\ldots,a_m\in \mg{F}$, $(m\geqslant 1)$, we denote by $\pff{a_1,\ldots,a_m}$ the symmetric bilinear form $\pff{a_1}\otimes\ldots\otimes\pff{a_m}$. We call any bilinear form isometric to  $\pff{a_1,\ldots,a_m}$  for some $a_1,\ldots,a_m\in F$ a  \emph{bilinear  $m$--fold Pfister form}. We consider  $\qf{1}_b$ as the $0$-fold bilinear Pfister form.
For any $m$-fold bilinear Pfister form $\delta$ we may write $\delta\simeq \qf{1}_b\perp \delta'$. If $\delta$ is anisotropic, then $\delta'$ is unique and called the \emph{pure part of $\delta$} (see \cite[p.36]{Elman:2008}).

By a \textit{quadratic form over $F$} we  mean a pair $(V,q)$ of a finite dimensional $F$-vector space $V$ and a map  $q:V\rightarrow F$ such that
 $q(\lambda x)=\lambda^2q(x)$ for all $x\in V$ and $\lambda\in F$, and such that $b_q:V\times V\rightarrow F,(x,y)\longmapsto q(x+y)-q(x)-q(y)$ is $F$--bilinear.
Then $(V,b_q)$ is a symmetric bilinear form over $F$, called the \emph{polar form of $(V,q)$}. If $(V,b_q)$ is nondegenerate then we call $(V,q)$ \emph{nonsingular},  and \emph{singular} otherwise.
By the \emph{radical of $(V,q)$} we mean the set 
$$\rad(V,q)= \{ x\in \rad(V,b_q)\mid q(x)= 0\}\,.$$
We say that $(V,q)$ is \emph{regular} if $\rad(V,q)=\{0\}$.

Consider a quadratic form $\rho=(V,q)$ over $F$. We call $\dim_F(V)$ the \emph{dimension of $\rho$} and denote it by $\dim(\rho)$.
 We say that $\rho$ is \emph{isotropic} if $q(x)=0$ for some $x\in V\setminus\{0\}$,  and in this case we call $x$ an \emph{isotropic vector of $\rho$}. Otherwise we say that $\rho$ is \emph{anisotropic}. We say that $\rho$ \emph{represents an element $a\in {F}^\times$} if there exists an $x\in V$ such that $q(x)=a$.
For $c\in {F}^\times$ let $c\rho$ denote the quadratic form $(V,cq)$, where $(cq)(x)=c(q(x))$ for all $x\in V$.
Let $\rho_1=(V,q)$ and $\rho_2=(W,q')$ be quadratic forms over $F$.
By an \emph{isometry of quadratic  forms $\rho_1\rightarrow\rho_2$} we mean an isomorphism of $F$--vector spaces $f:V\rightarrow W$ such that $q(x)=q'(f(x))$ for all $x\in V$. If such an isometry exists, we say $\rho_1$ and $\rho_2$ are \emph{isometric} and we write $\rho_1\simeq\rho_2$.
We say that $\rho_1$ is \emph{similar to $\rho_2$} if there exists $c\in F^\times$ such that $\rho_1\simeq c\rho_2$.

For $n\in\mathbb{N}$ and  $a_1,\ldots, a_n,b,c\in F$, we denote the quadratic form $(F^n,q)$ where $q:F^n\rightarrow F$ is given by $(x_1,\ldots, x_n)\mapsto \sum_{i=1}^n a_ix_i^2$ by $\qf{a_1,\ldots, a_n}$, and the
 quadratic form $(F^2,q')$ where $q':F^2\rightarrow F$ is given by $(x,y)\mapsto bx^2+xy+cy^2$ by $[b,c]$.

The following is an elementary result which is well known, but we include  a proof for convenience.

\begin{lemma}\label{cor:roundsim}
Let $\rho$ and $\rho'$ be $2$-dimensional quadratic forms over $F$ representing $1$. If $\rho$ is similar to $\rho'$ then $\rho\simeq\rho'$.
\end{lemma}
\begin{proof} Note that for all   non-zero elements $a$ represented by $\rho$ we have
$a\rho\simeq \rho$. If $\rho\simeq \qf{1,0}$, this is obvious, otherwise it follows from  \cite[(9.9) and (10.3)]{Elman:2008}.
If $\rho$ is similar to $\rho'$ then there exists an element $a\in F^\times$ such that $\rho\simeq a\rho'$. As $\rho'$ represents $1$, it follows that $\rho$ represents $a$. Therefore  we have that $\rho\simeq a\rho \simeq a^2\rho'\simeq \rho'.$\end{proof}

We say that $\rho_2=(W,q')$ is \emph{dominated by $\rho_1=(V,q)$} if there exists an injective map $t:W\lra V$ such that $q(t(x))= q'(x)$ for all $x\in W$. We say that $\rho_2$ is \emph{weakly dominated by $\rho_1$} if $\rho_2$ is dominated by a form similar to $\rho_1$.  
 Note that if $\rho_1$ and $\rho_2$ are of the same dimension, then $\rho_2$ being dominated by  $\rho_1$ is equivalent to $\rho_1\simeq \rho_2$ and similarly, in this case $\rho_2$ being weakly dominated  by $\rho_1$ is equivalent to $\rho_1$ and $\rho_2$ being similar.
 We say $\rho_2$ is a \emph{subform of $\rho_1$} if there exists a quadratic form $\rho_3$ over $F$ such that $\rho_1\simeq\rho_2\perp\rho_3$. Note that a subform of a quadratic form is also dominated by that quadratic form and that if $\rho_2$ is a nonsingular form dominated by  $\rho_1$ then $\rho_2$ is also a subform of $\rho_1$ (see \cite[(7.10)]{Elman:2008}).

By \cite[(7.31)]{Elman:2008}, for every quadratic  form $\rho$ over $F$, there exist elements $a_1,b_1,\ldots, a_n,b_n\in F$ and $c_1,\ldots,c_m\in F$ such that
\begin{eqnarray}\label{eqnarray:quadform} \rho\simeq [a_1,b_1]\perp\ldots\perp[a_n,b_n]\perp\qf{c_1,\ldots, c_m}\,.\end{eqnarray}
We say that a quadratic form $\rho$ as in (\ref{eqnarray:quadform}) is of \emph{type} $(n,m)$. 
 In this case, $m$ is the dimension of the  radical of the polar form of $\rho$, and $2n=\dim(\rho)-m$.
As isometries preserve the dimension and the radical of the polar form of a quadratic form, they also preserve the type. Moreover, as scaling a quadratic form does not change the type, similar quadratic forms must also have the same type.
Note  that $\rho$ is 
nonsingular if and only if $m=0$. If $n=0$ we say that $\rho$ is \emph{totally singular}. The form $ \qf{c_1,\ldots, c_m}$ is uniquely determined up to isometry  by $\rho$, and we call it the \emph{quasilinear part of $\rho$}. Note that $\rho$ is  regular if and only if $\qf{c_1,\ldots, c_m}$ is anisotropic. 
We use the following isometry,  which can be checked directly. For $a,b,c,d\in F$
we have  $$[a,b]\perp [c,d]\simeq [a+c, b] \perp [c,b+d].$$


 We call the quadratic form $[0,0]$ the \emph{hyperbolic plane} and denote it by $\HH$. 
Let  $\rho$ be a  quadratic form over $F$.  Then there exists an anisotropic  quadratic form $\rho'$,  and a nonnegative integers $n,m$ such that  $\rho\simeq \rho'\perp n\times \HH\perp m\times\qf{0}$.  In this decomposition, the integers $n,m$ are uniquely determined and  $\rho'$  is uniquely determined up to isometry (see  \cite[(2.4)]{HoffmannLaghribi:qfpfisterneigbourc2}). We  call $\rho'$ \emph{the anisotropic part of $\rho$}   and denote it  by $\rho_\an$.
We call the integer $n$ \emph{the Witt index of  $\rho$} and denote it by $i_W(\rho)$.

\begin{lemma}\label{lemma:subfom}
Let $\rho$ and $\psi$ be anisotropic  quadratic forms over $F$ such that $\rho$ is nonsingular. Then $\psi$ is dominated by $\rho$ if and only if $i_W(\rho\perp \psi)=\dim(\psi)$.
\end{lemma}
\begin{proof}
See \cite[(2.16)]{HoffmannLaghribi:Isoqfffquadricc2}.
\end{proof}

Let $\rho$ be a nonsingular quadratic form over $F$. Then we say that $\rho$  is \emph{hyperbolic} if $\dim (\rho)=2i_W(\rho)$, 
that is, $\rho$ is hyperbolic if  $\rho\simeq \frac{1}{2} \dim(\rho)\times \HH$. We say two nonsingular quadratic forms $\rho$ and $\rho'$ are \emph{Witt equivalent} if $\rho_\an\simeq\rho'_\an$. This is an equivalence relation on the set of nonsingular quadratic forms over a field $F$ (see  \cite[\S 8.A]{Elman:2008})



Let $\rho$ be as in (\ref{eqnarray:quadform}) and  nonsingular. Then the  class of $a_1b_1+\ldots+a_nb_n$ in $F/\wp(F)$ is an invariant of $\rho$ called \emph{the Arf invariant}  (see \cite[(13.5)]{Elman:2008}). We denote this invariant by $\Delta(\rho)$. Note that as for all $a,b\in F$ and $c\in F^\times$ we have that $c[a,b]\simeq [ca, c^{-1}b]$, similar nonsingular quadratic forms have the same Arf invariant. Moreover, two Witt equivalent nonsingular quadratic forms have the same Arf invariant  (see \cite[(\S13)]{Elman:2008}).



  Let $\delta=(V,b)$ be a nondegenerate  symmetric  bilinear form over $F$ and $\rho=(W,q)$ be a quadratic form over $F$. There is a natural  map $b\otimes q:V\otimes_F W\rightarrow F$ determined by the rule that 
$( b\otimes q) (v\otimes w)=  b(v,v) \cdot q(w)$
for all $w\in W, v\in V$, and $(V\otimes_F W, b\otimes q)$ is a quadratic form over $F$ with polar form $b\otimes b_q$, called the \emph{tensor product of $\delta$ and $\rho$} and denoted $ \delta\otimes\rho$.

{For $n\geq 1$ an integer,} by an $n$-fold Pfister form over $F$, we mean the tensor product of a $2$-dimensional nonsingular quadratic form over $F$ which represents $1$ with an $(n-1)$-fold bilinear Pfister form over $F$ (see \cite[\S 9]{Elman:2008}).
For  $a \in F$ and $b_1,\ldots, b_{n-1}\in F^\times$ ($n>1$), we denote the $n$-fold Pfister  form $\pff{b_1,\ldots, b_{n-1}}\otimes[1,a]$  by  $\pfr{b_1,\ldots, b_{n-1},a}$.  Pfister forms have the property that they are either anisotropic or hyperbolic (see \cite[(9.10)]{Elman:2008}).
Note that for all $n$-fold Pfister forms $\rho$ with $n>1$ we have  $\Delta(\rho)\in\wp(F)$. 

We call a quadratic form $\rho$ a \emph{Pfister neighbour} if there exist an $n$-fold Pfister form $\pi$ for some $n$ such that $\rho$ is weakly dominated by $\pi$ and $\dim(\rho)>2^{n-1}$. In this case $\pi$ is uniquely determined by $\rho$ up to isometry and $\rho$ is isotropic if and only if $\pi$ is hyperbolic (see \cite[(23.11)]{Elman:2008}).

\section{Quadratic splitting fields of {P}fister forms}

Let $\rho=(V,q)$ be a quadratic form over $F$ and
let $K/F$ be a field extension. Then we write $\rho_K=(V\otimes_F K, q_K) $, where the quadratic map $q_K:V\otimes_FK\rightarrow K$ is determined by $q_K(v\otimes k)=k^2q(v)$ for all $v\in V$ and $k\in K$. 
Note that if $K/F$ is a transcendental extension, then $\rho$ is anisotropic if and only if $\rho_K$ is isotropic (see  \cite[(7.15)]{Elman:2008}).

\begin{prop}\label{lemma:subs}
Let $\pi$ be an anisotropic $n$-fold  Pfister form over  $F$ and $F(\alpha)/F$ be a field extension such that $\alpha^2+\alpha=c\in F\backslash\wp(F)$. 
Then the following are equivalent.
\begin{itemize}
\item[($1$)] $\pi_{F(\alpha)}$ is hyperbolic.
\item[($2$)] $[1,c]$ is a subform of $\pi$.
\item[($3$)]  $\pi\simeq\pfr{b_1,\ldots, b_{n-1},c}$ for some $b_1,\ldots, b_{n-1}\in F^\times$.
\item[$(4)$] $i_W(\pi\perp [1,c])=2$. 
\end{itemize}
\end{prop}
\begin{proof}  For $(1)$ implies $(2)$, {see} \cite[(34.11)]{Elman:2008}.
For $(2)$ implies $(3)$, see \cite[Chapt. IV, (4.1)]{Baeza:1978}.
If $(3)$ holds, then $\pi_{F(\alpha)}$ is isotropic, and hence hyperbolic, and therefore $(1)$ holds.
That $(2)$ implies $(4)$ is clear. If $(4)$ holds, then $[1,c]$ is a subform of $\pi$ by 
Lemma \ref{lemma:subfom}, and hence $(4)$ implies $(2)$.
\end{proof}

\begin{prop}\label{lemma:reps}  
Let $\pi$ be an anisotropic   $n$-fold Pfister form over  $F$ and $d\in F\backslash F^2$. If $n=1$ then $\pi_{F(\sqrt{d})}$ is anisotropic. Otherwise, the following are equivalent.
\begin{itemize}
\item[($1$)] $\pi_{F(\sqrt{d})}$ is hyperbolic.
\item[($2$)] $\pi$ dominates  $\qf{1,d}$.
\item[($3$)]  $\pi\simeq\pfr{d,b_1,\ldots, b_{n-2},a}$ for some $a\in F$ and $b_1,\ldots, b_{n-2}\in F^\times$.
\item[$(4)$] $i_W(\pi\perp \qf{1,d})=2$.
\end{itemize}
\end{prop}
\begin{proof} For $n=1$ see \cite[(1.1)]{laghribi:qfdim6}.
Otherwise we may assume that $n>1$.
For $(1)$ implies $(3)$, by 
 \cite[(1.4)]{laghribi:wittkers}
we have that $\pi\simeq\pfr{d+x^2,b_1,\ldots, b_{n-2},a'}$ for some $a',x\in F$ and $b_1,\ldots, b_{n-2}\in F^\times$. As $\pff{b_1,\ldots, b_{n-2}}$ represents $1$, it follows easily from \cite[(15.6)]{Elman:2008} that $\pi\simeq\pfr{d,b_1,\ldots, b_{n-2},a}$ for some $a\in F$.
That $(3)$ implies $(2)$ is obvious.
If $(2)$ holds then $\pi_{F(\sqrt{d})}$ is isotropic and hence hyperbolic, and therefore $(1)$ holds.
The equivalence of $(3)$ and $(4)$ follows from Lemma \ref{lemma:subfom}.
\end{proof}

Proposition \ref{lemma:subs}  (respectively, Proposition \ref{lemma:reps}) says that an  $n$-Pfister form becomes hyperbolic over a separable (resp.~inseparable) quadratic extension if and only if it is a multiple of a corresponding $1$-fold Pfister form (resp.~$1$-fold bilinear Pfister form). In 
 \cite{fairve:thesis}, this idea is generalized, and the `linkage' of two $n$-fold Pfister forms by $m$-fold Pfister or bilinear Pfister forms ($m<n$) is considered.

\section{Function fields of quadratic forms}

In this section we collect the various results we need on the behaviour of quadratic forms over the function fields of quadratic forms.
Assume $\rho$ is regular.
If $\dim(\rho)\geqslant 3$ or if $\rho$ is  anisotropic of dimension $2$, then we call the function field of the projective quadric over $F$ given by $\rho$ the \emph{function field of $\rho$} and denote it by $F(\rho)$. In the remaining cases we set $F(\rho)=F$.
This agrees with the definition in \cite[Section 22]{Elman:2008}.  For a regular quadratic form $\rho$, the field extension $F(\rho)/F$ is purely transcendental if and only if $\rho$ is isotropic (see \cite[(22.9)]{Elman:2008}).

 The following theorem collects the main major results we will use, the subform Theorem and the $2$-power separation Theorem.

\begin{thm}\label{prop:subformthm}
Let  $\varphi$ and $\psi$ be anisotropic quadratic forms over $F$.
\begin{enumerate}[$(a)$]
\item If $\varphi_{F(\psi)}$ is hyperbolic then $\psi$ is weakly dominated by  $\varphi$.
\item If $\dim(\varphi)\leqslant 2^n< \dim(\psi)$ then $\varphi_{F(\psi)}$ is anisotropic.
\end{enumerate}
\end{thm}
\begin{proof}
See \cite[(22.9)]{Elman:2008} and \cite[(1.1)]{HoffmannLaghribi:Isoqfffquadricc2} respectively.
\end{proof}

\begin{prop}\label{prop:oneoverother}
 Let $L=F(\alpha)$ where $\alpha^2+\alpha=c$ for  $c\in F\backslash \wp(F)$ and let $K=F(\sqrt{d})$ for $d\in F\backslash F^2$. 
Let $\pi$ and $\pi'$ be anisotropic  $n$-fold Pfister forms such that $\pi_L$ and $\pi'_K$ are anisotropic and let $\rho= (\pi\perp[1,c])_\an$ and $\psi=(\pi'\perp\qf{1,d})_\an$.  Then $\rho_{F(\psi)}$ and $\psi_{F(\rho)}$ are anisotropic. 
\end{prop}
\begin{proof} Note first that by Lemmas \ref{lemma:subs} and \ref{lemma:reps} the forms $\rho$ and $\psi$ are of dimension $2^n$.
Assume the form $\rho_{F(\psi)}$ is isotropic.  As $[1,c]_L$ is hyperbolic and $\pi_L$ is anisotropic, it follows that $\rho_L\simeq \pi_L$.
 Then
as $\rho_{F(\psi)}$ is isotropic, it follows that  $\rho_{L(\psi)}\simeq \pi_{L(\psi)}$ is hyperbolic. 

Suppose $\psi_L$ is isotropic. Then, as $\qf{1,d}_{L}$ is anisotropic, $\psi_L$ is regular and hence $L(\psi)/L$ is a purely transcendental extension, and therefore  $\pi_L$ isotropic, a contradiction. Hence $\psi_L$ is anisotropic, and it  follows from Theorem \ref{prop:subformthm}, $(a)$ that $\psi_L$ is weakly  dominated by $\pi_L$.  As $\psi_L$ and $\pi_L$ are of the same dimension, if  $\psi_L$ is weakly  dominated by $\pi_L$ then  $\psi_L$ and $\pi_L$  are similar. However, as $\psi_L$ is of type $(2^{n-1}-1, 2)$ and $\pi_L$ is of type $(2^{n-1},0)$, this cannot occur.  Hence $\rho_{F(\psi)}$ is anisotropic.

Assume now that $\psi_{F(\rho)}$ is isotropic.  Let $b_1,\ldots,b_{n-1}\in F^\times$ and $a\in F$  be such that $\pi'\simeq\pfr{b_1,\ldots,b_{n-1},a}$ and let $\delta$ be the pure part of $\pff{b_1,\ldots,b_{n-1}}$. Then we have that $$\psi_K\simeq (\delta\otimes[1,a]\perp \qf{1})_\an\perp \qf{0}\,.$$ 
As $\qf{1,d}_{F(\rho)}$ is anisotropic by Theorem \ref{prop:subformthm}, $(b)$, we have that $i_W(\psi_{F(\rho)})\geqslant 1$. Hence 
$$\psi_{F(\rho)}\simeq \HH \perp \varphi\perp \qf{1,d}_{F(\rho)}$$
for some nonsingular quadratic form $\varphi$ over $L$ of dimension $2^n-4$.
Considering these forms over $K(\rho)$ gives
$$ \psi_{K(\rho)} \simeq  ( \HH \perp \varphi\perp \qf{1})_{K(\rho)}\perp \qf{0}\simeq (\delta\otimes[1,a] \perp \qf{1})_{K(\rho)}\perp \qf{0}\,.$$ 
As $( \HH \perp \varphi\perp \qf{1})_{K(\rho)}$ and $ (\delta\otimes[1,a] \perp \qf{1})_{K(\rho)}$ are regular,  \cite[(2.6)]{HoffmannLaghribi:qfpfisterneigbourc2} gives 
$$( \HH \perp \varphi\perp \qf{1})_{K(\rho)}\simeq (\delta\otimes[1,a] \perp \qf{1})_{K(\rho)}\,.$$ 
Note that $\delta\otimes[1,a]\perp\qf{1}$ is a Pfister neighbour of $\pi'$. Hence $\pi'_{K(\rho)}$ is hyperbolic. 

Suppose $\rho_K$ is isotropic. Then $K(\rho)/K$ is a purely transcendental extension and therefore  $\pi'_K$ is isotropic, a contradiction. Hence $\rho_K$ is anisotropic, and it then follows from Theorem \ref{prop:subformthm}, $(a)$ that $\rho_K$ is weakly  dominated by $\pi'_K$. In particular $\Delta(\rho_K)\in\wp(K)$ and hence $\Delta(\rho)\in\wp(F)$ as $K/F$ is an inseparable quadratic extension. However, $\Delta(\rho)= c\mod\wp(F)$ and $c\notin\wp(F)$ by assumption. Hence $\psi_{F(\rho)}$ is anisotropic. 
\end{proof}

\begin{prop}\label{prop:sepoversep} Let  $L=F(\alpha)$ where $\alpha^2+\alpha=c$ and $K=F(\beta)$ where $\beta^2+\beta=d$  for  $c,d\in F\backslash \wp(F)$.
Let $\pi$ and $\pi'$ be anisotropic $n$-fold Pfister  forms such that $\pi_L$ and $\pi'_K$ are anisotropic and let $\rho= (\pi\perp[1,c])_\an$ and $\psi=(\pi'\perp[1,d])_\an$.  If $\rho_{F(\psi)}$ is isotropic then $L\simeq_F K$.
\end{prop}
\begin{proof}By Proposition \ref{lemma:subs} $\rho$ and $\psi$ have dimension $2^n$. 
As $\pi_L$ is anisotropic and $[1,c]_L$ is hyperbolic, we have that $\rho_L\simeq \pi_L$. Further, as $\rho_{F(\psi)}$ is isotropic, we have that $\rho_{L(\psi)}\simeq\pi_{L(\psi)}$ is hyperbolic. If $\psi_L$ is isotropic then $L(\psi)/L$ is a transcendental extension and we have a contradiction with the anisotropy of $\pi_L$. Hence $\psi_L$ is anisotropic and it follows from Theorem \ref{prop:subformthm}, $(a)$ that $\psi_L$ is similar to a subform of $\pi_L$. 
As they are of the same dimension, $\psi_L$ must be similar to $\pi_L$.  In particular $\Delta(\psi_L)\in\wp(L)$. It follows that $\Delta(\psi)= c \mod\wp(F)$, and hence $[1,c]\simeq [1,d]$. Then $L\simeq_F K$.
\end{proof}

\begin{prop}\label{prop:insepoverinsep} Let  $L=F(\sqrt{c})$ and $K=F(\sqrt{d})$ where   $c,d\in F\backslash F^2$.
Let $\pi$ and $\pi'$ be anisotropic $n$-fold Pfister forms such that $\pi_L$ and $\pi'_K$ are anisotropic and let $\rho= (\pi\perp\qf{1,c})_\an$ and $\psi=(\pi'\perp\qf{1,d})_\an$.  If $\rho_{F(\psi)}$ is isotropic then $L\simeq_F K$.
\end{prop}
\begin{proof} By Proposition \ref{lemma:reps},  $\rho$ and $\psi$ have dimension $2^n$. 
Suppose $\rho_{F(\psi)}$ is isotropic.
Using similar arguments to those in Proposition \ref{prop:oneoverother} we have that $\pi_{L(\psi)}$ is hyperbolic. 

Assume that $\qf{1,d}$ is anisotropic over $L$. Then $\psi_L$ is regular. If $\psi_L$ is isotropic then $L(\psi)/L$ is a purely transcendental extension and $\pi_{L(\psi)}$ hyperbolic implies that $\pi_L$ is hyperbolic, a contradiction. Hence  $\psi_L$ is  anisotropic and it follows from Theorem \ref{prop:subformthm}, $(a)$ that $\psi_L$ is weakly dominated by $\pi_L$. 
As $\psi_L$ and $\pi_L$ are of the same dimension, if  $\psi_L$ is weakly  dominated by $\pi_L$ then  $\psi_L$ and $\pi_L$  are similar. However, as $\psi_L$ is of type $(2^{n-1}-1, 2)$ and $\pi_L$ is of type $(2^{n-1},0)$, this cannot occur. 
Hence $\qf{1,d}$ is isotropic over $L$ and therefore similar to $\qf{1,c}$. That  $\qf{1,c}\simeq \qf{1,d}$ follows from Lemma \ref{cor:roundsim}. Then $L\simeq_F K$.
\end{proof}

In the preceding propositions we have considered anisotropic quadratic forms of dimension $2^n$ that are the anisotropic part of the difference of an $n$-fold Pfister form and a $1$-fold Pfister form (or a $1$-fold quasi-Pfister form in the case of singular forms). These are examples of characteristic $2$ versions of the  so-called `Twisted Pfister forms' studied in \cite{Hoffmann:twisted}.

\section{Total linkage of {P}fister forms}\label{section:totlink}

Let $\pi$ and $\pi'$ be $n$-fold Pfister forms over $F$. 
 We say that
\begin{itemize}
\item $\pi$ is \emph{totally (separably or inseparably) linked to} $\pi'$ if for every (separable or inseparable) quadratic field extension $K/F$ such that $\pi_K$ is hyperbolic   $\pi'_K$ is also hyperbolic,
\item $\pi$ and $\pi'$ are \emph{totally (separably or inseparably) linked} if $\pi$ is totally (separably or inseparably) linked to $\pi'$ and vice versa.
\end{itemize}
We are interested in the following question.

\begin{question} For which fields $F$ are all totally linked Pfister forms isometric?
For which fields $F$ are  totally separably/inseparably  linked Pfister forms isometric?
\end{question}

Note first that for $1$-fold Pfister forms, these linkage questions are trivial, as it easily follows from   Lemmas \ref{lemma:subs} and \ref{cor:roundsim} that every anisotropic $1$-fold Pfister form $\pi$ is determined up to isometry by the unique separable quadratic extension that splits $\pi$, and every anisotropic $1$-fold Pfister form remains anisotropic after passing to an inseparable quadratic extension. {These questions} are also  of course also trivial for hyperbolic Pfister forms.

We now show that in general being totally (separably or inseparably) linked is not a symmetric relation (Corollary \ref{ex:insepnonsym}), that being totally inseparably linked is independent of being totally separably linked (Corollary \ref{ex:insep}), and that two totally linked Pfister forms  need not be isometric (Theorem  \ref{ex:iso}).

Let $\pi$ be an $n$-fold Pfister form over  $F$. We  consider the following sets:
\begin{eqnarray*}
S(\pi) &=& \{ K\mid K/F \textrm{ is a proper  separable field extension such that $\pi_K$ is hyperbolic} \}
\\
I(\pi)&=& \{  K\mid K/F \textrm{ is a  proper inseparable field extension such that $\pi_K$ is hyperbolic} \}\,.
\end{eqnarray*}
For two field extensions $L/F$ and $K/F$ we denote by $L\cdot K$ the field compositum of $L$ and $K$ over $F$, when it exists. In particular, if  $K=F(\rho)$ for a  non-totally singular anisotropic quadratic form $\rho$ of dimension at least $3$   then the compositum $L\cdot K$ exists and coincides with $L(\rho_L)$ by \cite[\S 4.2]{HoffmannLaghribi:qfpfisterneigbourc2}.

\begin{prop}\label{newpropsep} For any integer $n>1$, 
let $\pi$ be an anisotropic $n$-fold Pfister form over   $F$ and $K=F(\alpha)$ where $\alpha^2+ \alpha = c\in F\setminus \wp(F)$ such $K\notin S(\pi)$. Let   $\rho=(\pi\perp [1,c])_\an$.
Then $L=F(\rho)$ satisfies the following:
\begin{enumerate}[$(i)$]
\item For every anisotropic $n$-fold Pfister form $\sigma$ over $F$, $\sigma_L$ is anisotropic. 
\item   $L\cdot K\in S(\pi_L)$.
\item For all proper quadratic separable extensions $K'/F$ such that
 $K'\notin S(\pi)$ and $K'\not\simeq_F K$  we have $L\cdot K'\notin S(\pi_L)$.
\item For all proper quadratic inseparable extensions $K''/F$ such that  $K''\not\in I(\pi)$ we have  $L\cdot K''\notin I(\pi_L)$.
\end{enumerate}
\end{prop}
\begin{proof} 
Firstly, note that as $\pi_K$ is anisotropic, $\rho$  is of dimension $2^n$ over $F$ by Proposition \ref{lemma:subs}.
Further as $\Delta(\rho)\notin \wp(F)$ we have that $\rho$ is not similar to a Pfister form.
Hence for  every anisotropic $n$-fold Pfister form  $\sigma$ over $F$ we have that $\sigma_L$ is anisotropic by    Theorem \ref{prop:subformthm}, $(a)$, therefore $L$ has property  $(i)$.
Secondly note  that $\rho_L$ is clearly isotropic, $[1,c]_L$ is anisotropic by Theorem \ref{prop:subformthm}, $(b)$ and $\pi_L$ is anisotropic by property $(i)$. Hence we have that $i_W((\pi\perp[1,c])_L)=2$ and it follows from  Proposition \ref{lemma:subs} that $L$ has property  $(ii)$.

Let $c'\in F\backslash\wp(F)$ be such that for $K'=F(\alpha)$ where $\alpha^2+\alpha=c'$ we have $K'\notin S(\pi)$ and $K'\not\simeq_F K$. Let $\rho'=(\pi\perp [1,c'])_\an$. Then $\rho'_{L}$ is anisotropic by Proposition \ref{prop:sepoversep}, and hence $L\cdot K'\notin S(\pi_L)$ by Proposition \ref{lemma:subs}. 
Hence $L$ has property $(iii)$.
Similarly, it  follows from Proposition \ref{prop:oneoverother} and Proposition \ref{lemma:subs} that $L$ has property $(iv)$.
\end{proof}

\begin{prop}\label{newpropinsep} For any integer $n>1$, 
let $\pi$ be an anisotropic $n$-fold Pfister form over  $F$  and $K=F(\sqrt{d})$ where $d\in F\setminus F^2$ such $K\notin I(\pi)$. Let $\rho=(\pi\perp\qf{1,d})_\an$. Then  $L=F(\rho)$ satisfies the following:
\begin{enumerate}[$(i)$]
\item For every anisotropic $n$-fold Pfister form $\sigma$ over $F$, $\sigma_L$ is anisotropic. 
\item   $L\cdot K\in I(\pi_L)$.
\item For all proper quadratic inseparable extensions $K'/F$ such that $K'\notin I(\pi)$ and $K'\not\simeq_F K$  we have $L\cdot K'\notin I(\pi_L)$.
\item For all proper quadratic separable extensions $K''/F$ such that  $K''\not\in S(\pi)$ we have  $L\cdot K''\notin S(\pi_L)$.
\end{enumerate}
\end{prop}
\begin{proof} 
Firstly,  as $\pi_K$ is anisotropic, $\rho$  is of dimension $2^n$ over $F$ by Proposition \ref{lemma:reps}. 
Further, as $\rho$ is singular  we have that $\rho$ is not similar to a Pfister form.
Hence for  every anisotropic $n$-fold Pfister form  $\sigma$ over $F$ we have that $\sigma_L$ is anisotropic by    Theorem \ref{prop:subformthm}, $(a)$, therefore $L$ has property  $(i)$.
It  follows from {Propositions \ref{lemma:reps}, \ref{prop:oneoverother} and \ref{prop:insepoverinsep}} that $L$ has properties $(ii)-(iv)$ using arguments similar to those used in  Proposition \ref{newpropsep} for the corresponding properties.
\end{proof}

\begin{cor}\label{ex:insepnonsym} For any integer $n>1$, let $\pi$ and $\pi'$ be anisotropic $n$-fold Pfister forms over $F$ such that   $\pi'$ is  not totally separably (resp.~inseparably) linked to $\pi$. Then there exists a field extension $E/F$ such that $\pi_E$ is totally separably  (resp.~inseparably) linked to $\pi'_E$ but not vice versa.
\end{cor}
\begin{proof} 
Assume first that  $\pi'$ is  not totally separably linked to $\pi$.
Fix a field extension $M/F$ such that $\pi_M$ and $\pi'_M$ are anisotropic and there exists a field $K'=M(\alpha)$ where $\alpha^2+\alpha=c'\in M\setminus \wp(M)$ such that
$K'\in S(\pi'_M)$ and $K'\notin S(\pi_M)$.
Further assume that there exists a field $K=M(\beta)$ where $\beta^2+\beta=c\in M\setminus \wp(M)$ such that
$K\in S(\pi_M)$ and  $K\notin S(\pi'_M)$. Then by  Proposition \ref{newpropsep},  for the form $\rho_K=(\pi_M\perp[1,c])_\an$  and  the field $L_K=M(\rho_K)$
we have   that $\pi_{L_{K}}$ and  $\pi'_{L_{K}}$ are anisotropic  and
$L_{K}\cdot K\in S(\pi'_{L_K})$.

As $K'\notin S(\pi_M)$, by  Proposition \ref{lemma:subs} the form  $\varphi=(\pi_M\perp[1,c'] )_\an$ is of dimension $2^n$. 
It also follows from Proposition \ref{lemma:subs} that  for all field extensions $E/M$ with $c'\notin \wp(E)$ and $\pi_E$ anisotropic, 
we have that $\pi_{E\cdot K'}$ is hyperbolic if and only if $\varphi_{E}$ is isotropic.
As $K\not\simeq_M K'$ we have that $\varphi_{L_K}$ is  anisotropic by Proposition \ref{prop:sepoversep}. Hence   $L_{K}\cdot K'\notin S(\pi_{L_{K}})$.

We now build the field $E$ inductively.
Let $F=F_0$ and
let $F_{j+1}$ be the field  compositum  of the extensions $L_K$ for all $K\in S(\pi_{F_{j}})$ such that $K\notin S(\pi'_{F_{j}})$ as in the previous paragraph with $M=F_{j}$. Let $E=\bigcup_{i=1}^\infty F_i$. Then by construction   $\pi_E$ is totally separably linked to $\pi'_E$ but not vice versa.

Assume now that  $\pi'$ is  not totally inseparably linked to $\pi$.
Fix a field extension $M/F$ such that $\pi_M$ and $\pi'_M$ are anisotropic and there exists a field $K'=M(\sqrt{d'})$ where $d'\in M\setminus M^2$ such that
$K'\in I(\pi'_M)$ and $K'\notin I(\pi_M)$.
Further assume that there exists a field $K=M(\sqrt{d})$ where $d\in M\setminus M^2$ such that
$K\in I(\pi_M)$ and  $K\notin I(\pi'_M)$. Then by  Proposition \ref{newpropinsep}, 
for the form $\rho_K=(\pi_M\perp\qf{1,d})_\an$ over $M$  and the field $L_K=M(\rho_K)$
we have   that $\pi_{L_{K}}$ and  $\pi'_{L_{K}}$ are anisotropic and
$L_{K}\cdot K\in I(\pi'_{L_K})$.

As $K'\notin I(\pi_M)$,  Proposition \ref{lemma:reps} implies  that $\psi=(\pi_M\perp \qf{1,d'})_\an$ is of dimension $2^n$. 
It also follows from  Proposition \ref{lemma:reps} that  for all field extensions $E/M$ with $d\notin E^2$ such that $\pi_E$ is anisotropic,
we have that $\pi_{E\cdot K'}$ is hyperbolic if and only if $\psi_{E}$ is isotropic.
As $K\not\simeq_M K'$ the form  $\psi_{L_K}$ is anisotropic by Proposition \ref{prop:insepoverinsep}. Therefore   $L_{K}\cdot K'\notin I(\pi_{L_{K}})$. 

The required field  can now be constructed inductively similarly to the construction in the first part of the proof. 
\end{proof}

\begin{cor} \label{ex:insep}  For any integer $n>1$,  let $\pi$ and $\pi'$ be anisotropic $n$-fold Pfister forms over $F$ such that    $\pi$ and $\pi'$ are not totally separably (resp.~inseparably) linked. Then there exists a field extension $E/F$ such that $\pi_E$ and $\pi'_E$ are anisotropic, totally inseparably (resp.~separably) linked but not totally separably (resp.~inseparably) linked.
\end{cor}
\begin{proof} Assume first that $\pi$ and  $\pi'$ are not totally separably linked.
Fix a field extension $M/F$ such that $\pi_M$ and $\pi'_{M}$ are anisotropic  and there exists a field $K''\in S(\pi_M)$ but $K''\notin S(\pi'_M)$ and let $K''=M(\alpha)$ such that $\alpha^2+\alpha=c\in M\backslash\wp(M)$. Further assume that there exists a field $K\in I(\pi_M)$  but $K\notin I(\pi'_M)$ and let $K=M(\sqrt{d})$ for some $d\in M\backslash M^2$. 
Then by Proposition \ref{newpropinsep}, for the  form $\rho_K=(\pi'_M\perp \qf{1,d})_\an$  and  the field $L_K=M(\rho_K)$ we have that  $\pi_{L_{K}}$ and $\pi'_{L_{K}}$ are anisotropic, $L_{K}\cdot K\in I(\pi'_{L_{K}})$ and $L_K\cdot K''\notin S(\pi'_M)$.

 Now assume that there exists a field $K'=M(\sqrt{d'})$ where $d'\in M\setminus M^2$ such that $K'\in I(\pi'_M)$  and $K'\notin I(\pi_M)$. 
Then by Proposition \ref{newpropinsep}, for the  form $\rho'_{K'}=(\pi_M\perp\qf{1,d'})_\an$ and  the field $L'_{K'}=M(\rho')$ we have that  $\pi_{L'_{K'}}$ and $\pi'_{L'_{K'}}$ are anisotropic and $L'_{K'}\cdot K'\in I(\pi_{L'_{K'}})$.
As $K''\notin S(\pi'_M)$,  it follows from Lemma  \ref{lemma:subs}, that 
$\psi=(\pi'_M\perp [1,c])_\an$ is of dimension $2^n$. It also follows from Lemma  \ref{lemma:subs} that for all field extensions $E/M$ with $d'\notin E^2$ and $\pi'_E$ anisotropic, we have that $\pi'_{E\cdot K'}$ is hyperbolic  if and only if $\psi_{E}$ is isotropic. 
Hence by   Proposition \ref{prop:oneoverother}, $\psi_{L'_{K'}}$ is  anisotropic and therefore 
 and $L'_{K'}\cdot K''\notin S(\pi'_{L'_{K'}})$.

Consider the field compositum of
 field extensions $L_K/M$ and $L'_{K'}/M$  as above
for all $K$ such that  $K\in I(\pi_M)$ but   not in $I(\pi'_M)$,  and all $K'$ such that  $K'\in I(\pi'_M)$  but not in  $I(\pi_M)$.
Using this type of field compositum, we can inductively construct the required field extension $E/F$  in a similar manner to the construction in Corollary \ref{ex:insepnonsym}.
The statement on non-totally inseparably linked Pfister forms follows in a similar manner using Proposition  \ref{newpropsep}.
\end{proof}

\begin{lemma}\label{lemma:getseplink} 
 Let $\pi$ and $\pi'$ be anisotropic  $n$-fold Pfister forms over $F$ such that  $\pi$ and $\pi'$  are not isometric. Then there exists a field extension $K/F$ such that:
 \begin{enumerate}[$(i)$]
\item  $\pi_K\not\simeq\pi'_K$. 
\item $(\pi_K\perp\pi'_K)_\an$ is similar to an $n$-fold Pfister form.
\item For all anisotropic quadratic forms $\rho$ such that $\dim(\rho)\leqslant 2^n$ we have that $\rho_K$ is anisotropic. In particular, $\pi_K$ and 
${\pi'_K}$ are anisotropic.
\end{enumerate}
  \end{lemma}
\begin{proof} Let $F_0=F$ and $\rho_0= (\pi\perp\pi')_\an$. For $i\geqslant 1$, let $F_i=F_{i-1}(\rho_{i-1})$ and $\rho_{i}= ((\rho_{i-1})_{F_{i}})_\an\simeq ((\pi\perp\pi')_{F_{i}})_\an$. 
As both $\pi$ and $\pi'$ are nonsingular, it follows that $\rho_i$ is nonsingular, and in particular, of even dimension for all $i$. Moreover, 
as $\dim(\rho_{i+1})<\dim(\rho_{i})$ for all $i$ such that $\rho_{i}$ is nontrivial, there exists some $m>0$ such that $\dim(\rho_m)\neq0$ and  $\dim(\rho_i)=0$ for all $i>m$. In particular, the form $\rho_m$ is even dimensional and becomes hyperbolic over its own function field, and hence is similar to a Pfister form by \cite[(23.4)]{Elman:2008}.

As both $\pi$ and $\pi'$ represent $1$, we have that $\dim(\rho_0)<2^{n+1}$. Therefore by the Arason-Pfister Hauptsatz  \cite[(23.7)]{Elman:2008}, we have that  $2^n\leqslant\dim(\rho_i)<2^{n-1}$ for all $i\leqslant m$. Hence $\dim(\rho_m)=2^n$ and, again by the Hauptsatz, $\rho_m$ is similar to an  $n$-fold Pfister form over $F_m$. Hence $F_m$ has property $(i)$ and $(ii)$. 
 Moreover, as $\dim(\rho_i)>2^{n}$ for all $i\in\{0,\ldots, m-1\}$, we have that $F_m$ has property $(iii)$ by Theorem \ref{prop:subformthm}, $(b)$.
\end{proof}

\begin{thm}\label{ex:iso} For any integer  $n>1$,  let $\pi$ and $\pi'$ be anisotropic $n$-fold Pfister forms over $F$ such that    $\pi$ and $\pi'$ are not isometric. Then there exists a field extension $E/F$ such that $\pi_E$ and $\pi'_E$ are anisotropic, totally separably and inseparably linked but not isometric.
\end{thm}
\begin{proof}
By Lemma \ref{lemma:getseplink}, we may assume that $(\pi\bot\pi')_\an$ is similar to an $n$-fold Pfister form $\rho$ over $F$. For any field extension $K/F$ we have that $\pi_K\simeq \pi'_K$ if and only if $\rho_K$ is hyperbolic.

We may  construct a field extension $E/F$ in a similar  way to the construction of the fields in   Corollary \ref{ex:insep}   so that    $\pi_E$ and $\pi'_E$ are anisotropic  and totally  separably and inseparably linked. Further, by Proposition \ref{newpropsep}, $(i)$ and Proposition \ref{newpropinsep}, $(i)$, the Pfister form $\rho_E$ is anisotropic, and hence $\pi_E\not\simeq \pi'_E$.
\end{proof}

In Example \ref{ex:main} we give a pair of Pfister forms satisfying the hypotheses of Corollaries \ref{ex:insepnonsym}   and \ref{ex:insep} and Theorem \ref{ex:iso}. That is, Pfister forms that are neither separably nor inseparably linked, and hence also not isometric. This shows that examples of  Pfister forms with the linkage properties claimed to exist at the start of the section do indeed exist. In fact, 
the Pfister forms in Example \ref{ex:main} share no common quadratic extension of the base field over which they become hyperbolic, which is a much stronger property than needed to satisfy 
the hypotheses of Corollaries \ref{ex:insepnonsym}   and \ref{ex:insep} and Theorem \ref{ex:iso} We use the following easy lemma.
%

\begin{lemma}\label{lemma:1var}
Let $K=F(x)$, where $x$ is an indeterminate and $a\in F$. Further, let $\rho$ be a  quadratic form over $F$ and $\delta=\qf{a_1,\ldots,a_n}_b$ for some $a_1,\ldots, a_n\in F^\times$ be such that $\rho\perp\qf{a_1,\ldots,a_n,1}$ is anisotropic over $F$. Then $\varphi=\rho\perp\delta\otimes[1,x]\perp[1,a+x]$ is anisotropic over $F(x)$. 
\end{lemma}
\begin{proof} We may  set $T=x^{-1}$ and work  over $F(T)$. 
Suppose $\varphi$ is isotropic.  
By multiplying an isotropic vector of $\varphi$ by an appropriate power of  $T^2$, we can find 
an element $b\in F[T]$ represented by $\rho$ and 
$f_1,\ldots, f_{n+1}, g_1,\ldots, g_{n+1}\in F[T]$ such that $b, f_1,\ldots, f_{n+1}, g_1,\ldots, g_{n+1}$ are not all divisable by $T$ and 
$$b+ \sum_{i=1}^n a_i(f_i^2 +  f_ig_i +T^{-1}g_i^2) + f_{n+1}^2+f_{n+1}g_{n+1}+(T^{-1}+a)g_{n+1}^2=0\,.$$ 
Multiplying the above equation by $T$ and reducing modulo $T$, we obtain 
$$ \sum_{i=1}^n a_i(g_i(0))^2 +  (g_{n+1}(0))^2\,{=0}.$$
As $\qf{a_1,\ldots,a_n,1}$ is anisotropic, we get that $g_1,\ldots, g_{n+1}$  are divisible by $T$. Substituting $g_i= Th_i$ for $i=1,\ldots, {n+1}$ in the above equation and reducing modulo $T$ gives 
$$b(0)+ \sum_{i=1}^n a_i(f_i(0))^2 + (f_{n+1}(0))^2=0\,.$$ 
The anisotropy of $\rho\perp\qf{a_1,\ldots,a_n,1}$ now implies that $b,f_1,\ldots, f_{n+1}$ are also divisible by $T$, contradicting our hypothesis. Hence $\varphi$ is anisotropic.  \end{proof}

  \begin{ex}\label{ex:main} For $n>1$, let
 $K=F(x_1,\ldots, x_n, y_1,\ldots, y_n)$,  where
     $x_1,\ldots, x_n,y_1,\ldots, y_n$ are  indeterminates . Then $\pi_1=\pfr{x_1,\ldots, x_n}$ and  $\pi_2=\pfr{y_1,\ldots, y_n}$  are $n$-fold Pfister forms over  $K$ that 
  do not become hyperbolic  over any common separable or inseparable quadratic extension of $K$. In particular, they are not isometric nor totally separably or inseparably linked.
  \end{ex}
\begin{proof}  Note first that if $\pi_1$ and $\pi_2$ do not become hyperbolic over any common separable or inseparable quadratic extension of $K$, then it is obvious that they are not isometric nor totally separably or inseparably linked.
Using Lemma \ref{lemma:1var} and inducting on the number of variables $n$, we see that $\pi_i$  is anisotropic for $i=1,2$. Let $\delta_1$ be the pure part of $\pff{x_1,\ldots, x_{n-1}}$ and $\delta_2$ be the pure part of $\pff{y_1,\ldots, y_{n-1}}$ and let 
 $$\rho=\delta_1\otimes[1,x_n ] \perp\delta_2\otimes[1,y_n] \perp [1,x_n+y_n]\,.$$
Note that $\rho$ is Witt equivalent to $\pi_1\perp\pi_2$ and anisotropic, again using Lemma \ref{lemma:1var}.

Suppose $\pi_1$ and $\pi_2$ split over a common separable quadratic extension of $K$. Then it follows from Proposition \ref{lemma:subs} that there exist $w_1,\ldots, w_{n-1},z_1,\ldots, z_{n-1}\in K^\times$ and an element $a\in K$ such that $\pi_1\simeq\pfr{w_1,\ldots, w_{n-1},a}$ and $\pi_2\simeq \pfr{z_1,\ldots, z_{n-1},a}$. Let $\delta'_1$ be the pure part of $\pff{w_1,\ldots, w_{n-1}}$ and $\delta'_2$ be the pure part of $\pff{z_1,\ldots, z_{n-1}}$. Then 
$$  \pi_1\perp \pi_2 \simeq  \delta'_1\otimes[1,{a} ] \perp\delta'_2\otimes[1,{a}] \perp 2\times \HH$$
and hence $i_W(\pi_1\perp\pi_2)\geqslant 2$, contradicting the anisotropy of $\rho$. Hence $\pi_1$ and $\pi_2$ do not split over any common separable quadratic extension of $K$.

Suppose $\pi_1$ and $\pi_2$ split over a common inseparable quadratic extension of $K$. Then it follows from Proposition \ref{lemma:reps} that there exist $u_1,\ldots, u_{n-1},v_1,\ldots, v_{n-1}\in K^\times$ and an element ${d}\in K^\times$ such that $\pi_1\simeq\pfr{d, u_1,\ldots, u_{n-1}}$ and $\pi_2\simeq \pfr{d,v_1,\ldots, v_{n-1}}$.
 Let $\delta''_1$ be the pure part of $\pff{u_1,\ldots, u_{n-2}}$ and $\delta''_2$ be the pure part of $\pff{v_1,\ldots, v_{n-2}}$. Then 
\begin{eqnarray*}  \pi_1\perp \pi_2 &\simeq & \pff{d}\otimes\delta''_1\otimes[1,u_{n-1} ] \perp \pff{d}\otimes \delta''_2\otimes[1,v_{n-1}]\\ && \perp d[1,u_{n-1}] \perp d[1,v_{n-1}]  \perp [1,u_{n-1}+v_{n-1}]\perp \HH\,.
\end{eqnarray*}
The form $ d[1,u_{n-1}] \perp d[1,v_{n-1}] $ is clearly isotropic and hence
 $i_W(\pi_1\perp\pi_2)\geqslant 2$, contradicting the anisotropy of $\rho$. Hence $\pi_1$ and $\pi_2$ are not hyperbolic over any common separable quadratic extension of $K$. 
\end{proof}

Our definitions of a Pfister form totally, separably or inseparably linked to another Pfister form and of total, separable or inseparable linkage of two Pfister forms can be made independent of the characteristic of the field. Of course, over fields of characteristic different from $2$ there are no inseparable quadratic extensions, so our notions of inseparable linkage are trivial in this case. In particular, all notions of total separable linkage are equivalent to total linkage. 
It is straightforward to adapt the arguments needed for the proof of  Corollary \ref{ex:insepnonsym} and Theorem  \ref{ex:iso} to fields of characteristic different from $2$, substituting analogous notions, such as the discriminant for the Art invariant, where necessary. More precisely, one obtains the following results.

\begin{cor}\label{ex:insepnonsymcharnot2}  Let $K$ be a field of characteristic different from two. For any integer $n>1$, let $\pi$ and $\pi'$ be anisotropic $n$-fold Pfister forms over $K$ such that   $\pi'$ is  not totally   linked to $\pi$. Then there exists a field extension $E/K$ such that $\pi_E$ is totally  linked to $\pi'_E$ but not vice versa.
\end{cor}

\begin{thm}\label{ex:isocharnot2} Let $K$ be a field of characteristic different from two. For any integer  $n>1$,  let $\pi$ and $\pi'$ be anisotropic $n$-fold Pfister forms over $K$ such that    $\pi$ and $\pi'$ are not isometric. Then there exists a field extension $E/K$ such that $\pi_E$ and $\pi'_E$ are anisotropic, totally   linked but not isometric.
\end{thm}

\section{Linkage of Quaternion algebras}

\subsection{Quaternion algebras and $2$-fold Pfister forms}
In this section we apply the results of Section \ref{section:totlink} in the case of $2$-fold Pfister forms to give examples of quaternion algebras in characteristic $2$ analogous to those constructed in \cite[(2.1)]{GaribaldiSaltman:2010}. 
We refer to \cite{Draxl:1983} as a general reference on finite-dimensional algebras over fields.
By a \emph{quaternion $F$-algebra} we mean a $4$-dimensional central simple $F$-algebra. 
We say a quaternion  $F$-algebra is \emph{division} if for all $a,b\in Q$ we have that $a\cdot b=0$ if and only if $a$ or $b=0$. We also call non-division quaternion $F$-algebras \emph{split}.
Let $Q$ and $Q'$ be a quaternion  $F$-algebras. If $Q$ and $Q'$ are isomorphic we write $Q\simeq Q'$.

For a field extension $L/F$, we denote the quaternion $L$-algebra $Q\otimes_{{F}}L$ by $Q_L$.  
Let $K/F$ be a proper field extension, that is $K\neq F$.
We say that a quaternion $F$-algebra $Q$ \emph{contains  $K$} if there exists a subfield of $Q$  that is $F$-isomorphic to $K$. In this case  $K$ is necessarily a quadratic extension of $F$. Moreover $Q$ contains $K$ if and only if $Q_K$ is split by \cite[\S 14, Theorems 4 and 5]{Draxl:1983}.

Let $Q$ and $Q'$ be quaternion $F$-algebras. 
 We say that
\begin{itemize}
\item $Q$ is \emph{totally (separably or inseparably) linked to} $Q'$ if every (separable or inseparable) quadratic field extension of $F$ contained in $Q$ is also contained in  $Q'$,
\item $Q$ and $Q'$ are \emph{totally (separably or inseparably) linked} if $Q$ is totally (separably or inseparably) linked to $Q'$ and vice versa.
\end{itemize}

Any non-division quaternion $F$-algebra is isomorphic  to the $F$-algebra of two-by-two matrices over $F$. Hence all such  $F$-algebras are isometric and contain every quadratic field extension of $F$. Therefore our linkage properties are only of interest for division algebras.

We let $\Nrd_Q:Q\rightarrow F$ denote the reduced norm map (see \cite[\S 22, Def.~2]{Draxl:1983} for the definition). Considering $Q$ as an $F$-vector space, the pair $(Q,\Nrd_Q)$ is a $4$-dimensional quadratic form over $F$. In fact, $(Q,\Nrd_Q)$ is a $2$-fold Pfister form. This is easy to compute from the explicit basis for a quaternion algebra given in, for example, \cite[p25]{Knus:1998}.
Conversely, every $2$-fold Pfister form $\pi$ over $F$ can be associated with  a quaternion $F$-algebra $Q$ such that $\pi\simeq (Q,\Nrd_Q)$ via its Clifford algebra (see \cite[\S11 and \S12]{Elman:2008}).

\begin{prop}\label{lemma:splitting}
Let $Q_1$ and $Q_2$ be quaternion $F$-algebras and let $\pi_i=(Q_i,\Nrd_{Q_i})$ for $i=1,2$. 
 Then
 $Q_1\simeq Q_2$.
if and only if  $\pi_1\simeq \pi_2$.
In particular, $Q_i$ is split if and only if $\pi_i$ is hyperbolic.
\end{prop}
\begin{proof}
See  \cite[(12.5)]{Elman:2008}.
\end{proof}

We immediately get the following.

\begin{cor}\label{cor:obv}
Let $Q_1$ and $Q_2$ be quaternion $F$-algebras and let $\pi_i=(Q_i,\Nrd_{Q_i})$ for $i=1,2$. 
Then $Q_1$ is totally (separably or inseparably) linked to $Q_2$ if and only if $\pi_1$ is totally (separably or inseparably) linked to $\pi_2$. Moreover, $Q_1$ and $Q_2$ are totally (separably or inseparably) linked if $\pi_1$ and $\pi_2$ are totally (separably or inseparably) linked. 
\end{cor}

Hence the  results of Section \ref{section:totlink} directly give  constructions of  totally linked division quaternion $F$-algebras that are not isomorphic, totally separably linked but not inseparably linked quaternion division algebras, and so on.

In \cite{Meyer:infgen}, it is shown that the construction of totally linked non-isomorphic division quaternion algebras from \cite{GaribaldiSaltman:2010} can be adapted to give examples of  fields  of characteristic different from $2$ over which there are infinitely many non-isomorphic totally linked division quaternion algebras. It is straightforward to adapt the results in  Section \ref{section:totlink}
 with methods similar to  \cite{Meyer:infgen} in order to  give fields of characteristic $2$ over which there are infinitely many totally linked non-isometric $n$-fold Pfister forms, and to construct fields over which there are infinitely many totally separably (resp.~ inseparably) linked $n$-fold Pfister forms that are not inseparably linked (resp.~separably linked), and so on. Applying Proposition \ref{lemma:splitting} in the case of $2$-fold Pfister forms to these fields gives examples of fields of characteristic $2$ with infinitely many non-isomorphic totally linked division quaternion algebras. 
 In particular,  we get the following result. We leave the details of the proof to the reader. 
 
 \begin{prop}
 Let $F$ be a field such that there exist infinitely many non-isomorphic division  quaternion $F$-algebras. Then there exists a field extension $K/F$
such that there exist infinitely many totally linked division quaternion $K$-algebras that are not isomorphic.
  \end{prop}

For the rest of the section we consider particular fields over which all $n$-fold Pfister forms are hyperbolic for $n>2$, and hence linkage properties are uninteresting for these forms. Over these fields,  linkage properties are still of interest for $2$-fold Pfister forms, and hence for quaternion algebras. We state our main results in term of $2$-fold Pfister forms, but the translation to results on quaternion algebras is easy using Proposition \ref{lemma:splitting}.

\subsection{Fields with a unique inseparable quadratic extension}

We call $F$, a field of characteristic $2$,  \emph{perfect} if $F^2=F$.
As all $2$-fold Pfister forms are hyperbolic over perfect fields, and all quaternion algebras split, questions on linkage are uninteresting over such fields.

In this section we consider fields with $[F:F^2]=2$, that is, fields with a unique inseparable  quadratic extension. Note that $[F:F^2]\leqslant 2$ is equivalent to all $3$-dimensional totally singular quadratic forms being isotropic over $F$. Over such fields  all $n$-fold Pfister forms are hyperbolic for $n>2$. 
Moreover,  over such fields all $2$-fold Pfister forms  are clearly totally inseparably linked. However, it is not the case that any  totally separably linked Pfister forms are isometric  over all such fields, as we now show.

We need the following lemma, which constructs a field using the  compositum of function fields of totally singular quadratic forms. As regular totally singular forms may not be regular when extended to the function field of another totally singular form, these constructions are more subtle than those in Section \ref{section:totlink}. As such, we give a detailed proof (in the spirit of \cite[\S5]{mammone:u2}) to avoid any confusion with such a construction.

\begin{lemma}\label{lemma:nearlysquareclosing}
Let $F$ be a field. Then there exists a field extension $K/F$ with  $[K:K^2]=2$ such that for all non-isometric  $2$-fold Pfister forms $\pi$ and $\pi'$ over $F$, the forms $\pi_K$ and $\pi'_K$ are not isometric. 
\end{lemma}
\begin{proof} 
First note that if  $[F:F^2]=1$, that is, $F$ is perfect,  we may take $K=F(x)$, where $x$ is an indeterminate, as $[K:K^2]=2$ and the condition on the Pfister forms is trivial. 

Otherwise, let  $S_0$ be the set of all totally singular quadratic forms of dimension $3$ over $F$. Choose a well-ordering on the set $S_0$ and index its elements by ordinal numbers. So for some ordinal $\alpha$, we have $S_0=\{\psi_i\mid i<\alpha\}$. We construct a field $F^1$ by transfinite induction as follows: let $F_0=F$ and define 
\begin{itemize}
\item $F_k=F_{k-1}(\psi_k)$  if $k$ is not a limit ordinal and $\psi_k$ is anisotropic over $F_{k-1}$,
\item  $F_k=F_{k-1}$  if $k$ is not a limit ordinal and $\psi_k$ is isotropic over $F_{k-1}$,
\item $F_k=\bigcup_{j<k}F_j$  if k is a limit ordinal.
\end{itemize}
We then set $F^1=F_\alpha$.
For all $k$, if $\rho$ is  {an anisotropic} nonsingular quadratic form over $F_{k}$ such that $\rho_{F_{k+1}}$ is hyperbolic then either $\rho$ is hyperbolic or Witt equivalent to an orthogonal sum of forms similar to $3$-fold Pfister forms  over $F_{k}$  by \cite[(1.4)]{laghribi:wittkers}. In particular, for all $2$-fold Pfister forms $\pi$ and $\pi'$ over $F_{k}$ the form $(\pi\perp\pi')_{F_{k+1}}$ is hyperbolic if and only if $\pi\perp\pi'$ is hyperbolic or   Witt equivalent to a sum of $3$-fold Pfister forms over $F_{k}$. As $\dim(\pi\perp\pi')_{{\rm an}}{\leq} 6$,  if  $\pi\perp\pi'$ is Witt equivalent to a sum of $3$-fold Pfister form then it is hyperbolic by the Arason-Pfister Hauptsatz  \cite[(23.7)]{Elman:2008}. Hence $\pi_{F_{k+1}}\simeq \pi'_{F_{k+1}}$ if and only if  $\pi\simeq\pi'$ over $F_{k}$.
Hence it follows by transfinite induction that for all non-isometric  $2$-fold Pfister forms $\pi$ and $\pi'$ over $F$, the forms $\pi_{F^1}$ and $\pi'_{F^1}$ are non-isometric.

Let $S_1$ be the set of  all totally singular quadratic forms of dimension $3$ over $F^1$ and construct $(F^1)^1=F^2$ by the same procedure. Repeating this process, for $n\geqslant 1$, let $F^n={(F^{n-1})}^1$ and let $K=\bigcup_{n=1}^\infty F^n$. 
As there are no anisotropic totally singular forms of dimension $3$ over $K$, we must have the $[K:K^2]\leqslant 2$.  If there are no anisotropic $2$-fold Pfister forms over $F$, then we are done. Otherwise, by the construction, 
we have that for all non-isometric  $2$-fold Pfister forms $\pi$ and $\pi'$ over $F$, the forms $\pi_K$ and $\pi'_K$ are not isometric by  same argument given above for $F^1$. 
In particular, 
for all anisotropic $2$-fold Pfister forms $\pi$ over $F$ we have that $\pi_K$ is anisotropic, and hence $K$ is not perfect. Therefore $[K:K^2]=2$. 
\end{proof}

\begin{thm}\label{ex:isofnear2close}  Let $\pi$ and $\pi'$ be  $2$-fold Pfister forms over  $F$ such that   $\pi$ and $\pi'$ are not isometric. Then there exists a field extension $E/F$ with $[E:E^2]=2$ and such that $\pi_E$ and $\pi'_E$ are anisotropic and totally separably linked but not isometric.
\end{thm}
\begin{proof}
As at least one of $\pi$ and $\pi'$ is anisotropic, we have that $F$ is not perfect. We may then construct a field extension $K/F$ as in Corollary \ref{ex:insep} such that $\pi$ and $\pi'$ are totally separably linked but not isometric.  Using Lemma \ref{lemma:nearlysquareclosing}, we can then find a field extension $L/K$ with  $[L:L^2]=2$ such that 
$\pi_L\not\simeq\pi'_L$.  If $\pi_L$ and $\pi'_L$ are totally separably linked, then we are done. Otherwise we repeat the above process inductively to construct the required field $E$. 
\end{proof}

We now give two classes of fields of characteristic $2$  over which the separable linkage of two $n$-fold Pfister forms implies that they are isometric.

\subsection{Pfister forms over $F((t))$ with $F$ perfect}\label{subsection:laurent}

For the rest of this subsection, let  $F$ be a perfect field of characteristic 2. In this subsection, we consider the field of formal Laurent series over $F$, denoted $F((t))$.

\begin{prop}\label{prop:allsame} For every   $2$-fold Pfister form $\pi$ over  $K=F((t))$ we have that $\pi\simeq \pfr{t,a}$ for some $a\in F$.
\end{prop}
\begin{proof}
See \cite[Chapt.~ XIV, \S 5, Prop.~12]{serre1979local} for the result in terms of quaternion algebras. 
\end{proof}

\begin{prop}\label{prop:quatlaur}
Let $K=F((t))$ and take $a,b\in F$. Then $\pfr{t,a}\simeq \pfr{t,b}$ over $K$ if and only if $[1,a]\simeq [1,b]$ over $F$.
\end{prop}
\begin{proof}
If $[1,a]\simeq [1,b]$ then clearly $\pfr{t,a}\simeq \pfr{t,b}$.
Assume that $\pfr{t,a}\simeq \pfr{t,b}$. Then  $\pfr{t,a}\perp \pfr{t,b}$ is hyperbolic and it  then follows from   \cite[(19.5)]{Elman:2008} that $[1,a]\perp [1,b]$ is hyperbolic. Hence $[1,a]\simeq [1,b]$ by Witt cancellation, \cite[(8.4)]{Elman:2008} . 
\end{proof}

\begin{thm}\label{thm:main} 
Let $\pi$ and $\pi'$ be anisotropic $2$-fold Pfister forms  over $K=F((t))$. If there exists  an element $c\in F$  with  $c\notin \wp(F)$  such that  for $L=K(\alpha)$, where $\alpha^2+\alpha=c$, we have that  $\pi_L$ and $\pi'_L$ are hyperbolic, then $\pi\simeq \pi'$. \end{thm}
\begin{proof}
By Proposition \ref{prop:allsame}, $\pi\simeq \pfr{t,a}$ and $\pi'\simeq \pfr{t,b}$ for some $a,b\in F$. Assume that  $\pi_L$ and $\pi'_L$ are hyperbolic. Then $[1,a+c]\perp t[1,a]$ is   isotropic by Proposition \ref{lemma:subs}.
 It then follows from  \cite[(19.5)]{Elman:2008}  that at least one the quadratic forms $[1,a+c]$ and $[1,a]$ must be isotropic. We have that  $[1,a]$ is anisotropic by the assumption that $\pi$ is anisotropic, hence
  $[1,a+c]$ is isotropic. It then follows from  \cite[(13.2)]{Elman:2008} that $a=c+ \wp(F)$. Arguing similarly, we see that $b=c+ \wp(F)$.
 Hence $[1,a]\simeq d[1,b]$ for some $d\in F^\times$  and therefore
  $[1,a]\simeq [1,b]$ by Lemma \ref{cor:roundsim}. Hence $\pi\simeq \pi'$. \end{proof}

\begin{cor}
Let $\pi$ and $\pi'$ be anisotropic $2$-fold Pfister forms over $K=F((t))$. If $\pi$ and $\pi'$ are totally separably linked, then $\pi\simeq \pi'$.
\end{cor}
\begin{proof} By Proposition \ref{prop:allsame}, for every $2$-fold Pfister form $\pi$ over $K$ there exists an element $a\in F\setminus \wp(F)$ such that for $L=K(\alpha)$ where  $\alpha^2+\alpha=a$, we have $\pi_L$ is hyperbolic. The result then follows from Theorem \ref{thm:main}.
\end{proof}

\begin{remark}
Note that the argument in Theorem \ref{thm:main} requires that the separable extension linking the two $2$-fold Pfister forms $\pi$ and $\pi'$ over  $K=F((t))$ to be defined over $F$. In general, this condition cannot be weakened to linkage over $K$. That is, if $\pi$ and $\pi'$  share a separable quadratic extension of $K$, then they are not necessary isometric in general.  

Indeed, 
if separably linked did imply isometric for any two $2$-fold Pfister forms over $K$, then by Corollary \ref{cor:obv}, every pair of separably linked quaternion $K$-algebras would be isomorphic.
By \cite[p.258]{Draxl:quatsep} this implies  there would only be one division quaternion $K$-algebra up to isomorphism, and hence only one isometry class of anisotropic $2$-fold Pfister forms, and this is not always the case.
In fact it easily follows from Propositions \ref{prop:allsame}  and \ref{prop:quatlaur}  that there is a one-to-one correspondence between separable quadratic extensions of a perfect field and isomorphism classes of quaternion algebras over the field of Laurent series over the same perfect field.
\end{remark}

\subsection{Pfister forms over global fields}

Recall, as stated in the introduction of \cite{GaribaldiSaltman:2010}, that totally linked  quaternion algebras  are always isomorphic over global fields of arbitrary characteristic. This result is well-known to experts, but as no proof appears in the literature as yet,  we give a proof  in characteristic $2$ that shows that  totally separably linked $2$-fold Pfister forms over a global field are isometric. 

By a global field we mean a finite extension of $\mathbb{F}_{2^n}(t)$, where $t$ is a variable and  $\mathbb{F}_{2^n}$ is the field with $2^n$ elements for some $n\in \mathbb{N}$.
Recall that for any field $F$  with $[F:F^2]=2$ and a finite field extension $L/F$  we have that $[L:L^2]=2$.  Hence any global field $F$ has a unique inseparable quadratic extension, and thus all $2$-fold Pfister forms over $F$ are totally inseparably linked.

\begin{thm}\label{thm:global}
Let $F$ be a global field of characteristic $2$ and let $\pi$ and $\pi'$ be anisotropic $2$-fold Pfister forms over  $F$ that are  totally separably linked. Then $\pi\simeq \pi'$.
\end{thm}
\begin{proof}
Suppose $\pi$ and $\pi'$ are not isometric.
Then, without loss of generality, by the local-global principle of  Albert-Brauer-Hasse-Noether, \cite[(8.1.17)]{neukirch2013cohomology} and Proposition \ref{lemma:splitting},
 there is a prime $p$ of $F$ such that  $\pi_{F_p}$ is hyperbolic but $\pi'_{F_p}$ is not, where $F_p$ is the completion of $F$ with respect to the associated discrete valuation $v_p$.
Let $q_1,\dots,q_n$ be all the primes for which $\pi_{F_{q_i}}$ is anisotropic.
By the  approximation theorem, \cite[(2.4.1)]{Prestel:2005}, there exists $a \in F$ with $v_p(a)=1$ and $v_{q_i}(a)=-1$ for all $i\in \{1 \ldots , n\}$.
It follows  by Hensel's Lemma, \cite[(4.1.3)]{Prestel:2005}, that  $a \in \wp(F_p)$, while $a \not \in \wp(F_{q_i})$ for all  $i\in \{1 \ldots , n\}$.

Let $L=F(\alpha)$ where $\alpha^2+\alpha=a$. Then $L \cdot F_p=F_p(\alpha)=F_p$,  and therefore $L$ is a subfield of $F_p$.
Consequently, $\pi'_L$ is not hyperbolic.
Conversely,  $L \cdot F_{q_i}$ is a quadratic extension of $F_{q_i}$, and  $\pi_{L\cdot F_{q_i}}$ is hyperbolic. For all primes $r$ of $F$ other than $q_i$, $\pi_{F_r}$ is hyperbolic. Therefore, $\pi_L$ is locally hyperbolic everywhere, and hence hyperbolic  by \cite[(8.1.17)]{neukirch2013cohomology}.   Hence  $\pi$ and $\pi'$ are not totally separably linked.
\end{proof}

Note that it is easy to adapt the above proof to give the same result for fields of characteristic different from $2$.

\end{document}